\documentclass[a4paper,10pt]{amsart}
\usepackage[utf8]{inputenc}
\usepackage[T1]{fontenc}
\usepackage{amsfonts}
\usepackage{amsthm}
\usepackage{amsmath}
\usepackage{tikz}
\usepackage[new]{old-arrows}
\usepackage{pythonhighlight}
\usetikzlibrary{cd, positioning, arrows.meta, decorations.pathmorphing, shapes.multipart,
decorations.markings, calc}
\usepackage{etoolbox}\apptocmd\appendix{\pretocmd\section{\clearpage}{}{}}{}{}
\usepackage{hyperref}

\newtheorem{theorem}{Theorem}[section]
\newtheorem{lemma}[theorem]{Lemma}
\theoremstyle{definition}
\newtheorem{example}[theorem]{Example}

\newtheorem{ntc}[theorem]{Notation}

\newcommand{\CC}{\mathbb{C}}
\newcommand{\RR}{\mathbb{R}}

\newcommand{\PP}{\mathbb{P}}

\newcommand{\ZZ}{\mathbb{Z}}
\newcommand{\proy}[1]{\pi_{#1}}

\numberwithin{equation}{section}

\title[Retraction of the complement of a smooth curve]{Retraction of the complement of smooth projective hypersurfaces to an $n$-dimensional $\Delta$-complex.}
\author[E. Artal]{Enrique Artal Bartolo} 
\address{Departamento de Matem\'{a}ticas, IUMA \\
Universidad de Zaragoza \\
C.~Pedro Cerbuna 12, 50009, Zaragoza, Spain}
\urladdr{\url{http://riemann.unizar.es/~artal}}
\email{\href{mailto:artal@unizar.es}{artal@unizar.es}}
\author[A. Larraya]{Alba Larraya Sancho}
\address{Basque Center for Applied Mathematics\\
Alameda de Mazarredo 14,
48009 Bilbao, Spain}
\email{\href{mailto:alarraya@bcamath.org}{alarraya@bcamath.org}}
\author[M.Á. Marco]{Miguel \'A. Marco-Buzunáriz}
\address{Departamento de Matem\'{a}ticas, IUMA \\
Universidad de Zaragoza \\
C.~Pedro Cerbuna 12, 50009, Zaragoza, Spain}
\urladdr{\url{http://riemann.unizar.es/~mmarco}}
\email{\href{mailto:mmarco@unizar.es}{mmarco@unizar.es}}

\thanks{The first and third named authors are partially supported by  PID2024-156181NB-C33 funded by
MICIU/AEI/10.13039/501100011033. The second named author is partially supported by  PID2024-156181NB-C31 funded by
MICIU/AEI/10.13039/501100011033 and CEX2021-001142-S-20-9.}

\tikzcdset{%
   harpoonlurd/.code={\pgfsetarrows{tikzcd right to-tikzcd right to}},
   harpoonldru/.code={\pgfsetarrows{tikzcd left to-tikzcd left to}},
   harpoonluru/.code={\pgfsetarrows{tikzcd right to-tikzcd left to}},
   harpoonldrd/.code={\pgfsetarrows{tikzcd left to-tikzcd right to}},
}
\begin{document}

\begin{abstract}
We give an explicit strong deformation retraction of the complement of a Fermat hypersurface of degree $d$ in the complex projective space of dimension $n$ into a $\Delta$-complex of real dimension~$n$.

This retraction is performed in several steps, using the branched cover structure of the Fermat hypersurfaces to the degree $1$ case. Moreover, an implementation of the method in \texttt{Sagemath} is provided.
\end{abstract}

\maketitle
\section*{Introduction}
A result of Andreotti and Frankel \cite{andreotti1959} asserts that any affine smooth variety of complex dimension~$n$ has the homotopy type of a CW-complex of real dimension~$n$. Using the Segre embedding, it is not hard
to check that the complement of a hypersurface in $\PP^n$ is isomorphic to an affine variety. 
However, no general method has been given to explicitly compute this complex.
Since there are hypersurfaces with equations in a number field $K$ whose embedding depend on the embedding $K\hookrightarrow\CC$ (even up to homotopy)~\cite{acgm}, a purely algebraic method cannot exist.

Possibly the only known method applies to the following particular case.
Given $C$ an affine curve in $\CC^2$, Libgober gave in \cite{libgober1986} a construction of a $2$-dimensional CW-complex that is a deformation retract of its complement, using braid monodromy which is not an algebraic construction. In the previous language a complement of an affine curve in $\CC^2$ is the complement of a projective curve~$C$ in $\PP^2$, where $C$ contains a line $L$ (which can be seen as the line at infinity). 

However, it is not clear how to extend that construction to the projective case, that is a curve in $\CC\PP^2$ such that no irreducible component is a line, e.g. a smooth curve of degree $>1$.

In this work, we give a different construction to produce a specific $\Delta$-complex that is a deformation retract of the complement of a $d$-Fermat hypersurface in $\CC\PP^n$. Since any smooth hypersurface is isotopic to a Fermat hypersurface, this particular case gives a construction for the more general case of smooth hypersurfaces.

The actual construction starts with the trivial case $d=1$. In this case the complement is contractible and a point would suffice as $\Delta$-complex. Actually we will take a simplex of real dimension~$n$ which is also contractible.
The main point is to understand the lifting of this complex by the Kummer cover $[x_0:\dots:x_n]\mapsto[x_0^d:\dots:x_n^d]$.

\section{Smooth hypersurfaces}
Given a smooth hypersurface of degree $d$ in $\PP^n$, we know it can be moved to a hypersurface of the form
\begin{equation*}
	C_d=\left\{[x_0:\dots:x_{n}]\in\CC\PP^n\, \middle|\, \sum x^d_i=0\right\}
\end{equation*}
by an ambient isotopy; $C_d$ is called a \emph{Fermat hypersurface}.
Therefore, we will work with $C_d$, from now on.

We will denote:
\begin{eqnarray*}
	M_d&=&\left\{(x_0,\dots,x_n)\in\CC^{n+1}\, \middle|\, \sum x^d_i=1\right\} \\
	S_d&=&\left\{(x_0,\dots,x_n)\in M_d\,\middle|\, x^d_i\in\RR_{\geq0} \forall i\right\} \\
	K &=&\pi(S_d)
\end{eqnarray*}
where $\proy{}:\CC^{n+1}\setminus\{\mathbf{0}\}\to\PP^n$ is the quotient map.
We then have a commutative diagram:
\begin{equation*}
\begin{tikzcd}
S_d\rar[hookrightarrow]\dar["\proy{}"]& M_d\dar["\proy{}"]\\ 
K\rar[hookrightarrow] & \PP^n\setminus C_d.
\end{tikzcd}
\end{equation*}
We are going to see that the inclusion $i:S_d\longrightarrow M_d$ induces a homotopy equivalence.

We will start by doing the construction for the case $d=1$, since the general case will be constructed on top of it.

\begin{ntc}
For a complex number $x\in\CC$, we denote its real and imaginary parts by $x^r,x^i$, respectively,
i.e. $x^r,x^i\in\RR$ and $x=x^r + i x^i$. This notation also applies to vectors in $\CC^{n+1}$.
\end{ntc}


We will now define a strong retraction map 
\begin{equation}
\begin{tikzcd}[row sep=0pt,/tikz/column 1/.append style={anchor=base east}, /tikz/column 2/.append style={anchor=base west}]
M_1\times{[0,1]}\rar["r_1"]&M_1\\
(P, t)\rar[mapsto]& t P^r + (1 - t) P^i=:r_{1t}(P);
\end{tikzcd}
\end{equation}
note that $r_{10}$ is the identity,$r_{11}(M_1)=M_1\cap\RR^{n+1}$
and the restriction of $r_{1t}$ to $M_1\cap\RR^{n+1}$ is the inclusion.


Note that for $(x_0,\dots,x_n)\in M_1$, we already have that
\begin{align*}
	\sum_{j=0}^n x^r_j & =  1 \\
	\sum_{j=0}^n  x^i_j & =  0 \\
\end{align*}
and therefore $r_1(x_0,\dots,x_n,t)\in M_1$ for all $t\in[0,1]$, and in particular, we get that $r_1(x_0,\dots,x_n,1)=(x^r_0,\dots,x^r_n)\in M_1\cap\RR^{n+1}$ for all $(x_0,\dots,x_n)\in M_1$.
Let us summarize it.

%

\begin{lemma}
The map $r_1$ is a strong deformation retraction from $M_1$ to $M_1\cap\RR^{n+1}$.
\end{lemma}

Now we will define a strong deformation retraction
\[
r_2:(M_1 \cap \RR^{n+1})\times[0, 1]\to M_1 \cap \RR^{n+1} 
\]
from $M_1 \cap \RR^{n+1}$ to $S_1$.
Note that $M_1\cap\RR^{n+1}$
is a hypersurface in $\RR^{n+1}$, in Figure~\ref{fig:triangulo} we can see the example $n=2$. 

\begin{lemma}
	The space $S_1$ is a $n$-simplex.
\end{lemma}
\begin{proof}
	Let $e_i=(0,\dots,1,\dots,0)$ be the vectors of the canonical base. Then
	\begin{align*}
	S_1&=\left\{(x_0,\dots,x_n)\in\RR^{n+1}\,\middle|\,\sum_{i=0}^n x_i=1, x_i\geq1\right\} \\
	&=\left\{x_0e_0+\dots+x_ne_n\,\middle|\,\sum_{i=0}^n x_i=1, x_i\geq1\right\}=\Delta_n
	\qedhere
	\end{align*}
\end{proof}

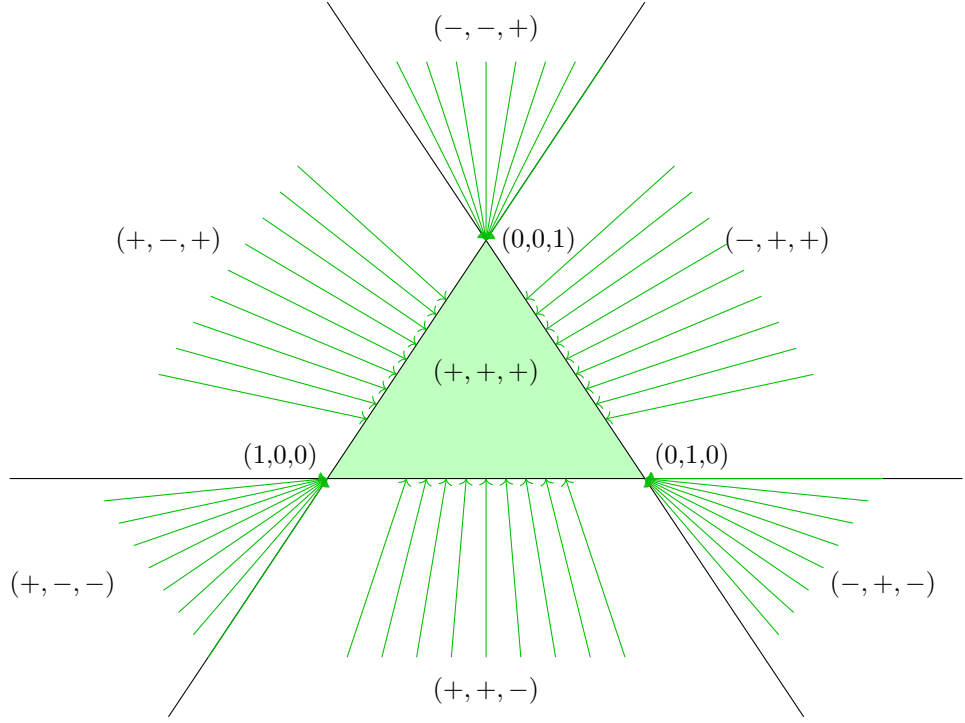
\begin{figure}[ht]
\begin{center}
\begin{tikzpicture}[scale=0.7,every text node part/.style={align=center}]
	\coordinate (A) at (-3,-4.5);
\coordinate (B) at (3,-4.5);
\coordinate (C) at (0,0);
\fill[green!25!white] (C) to (B) to (A);
	\draw ($2*(C)-(A)$) to ($2*(A)-(C)$);
	\draw ($2*(B)-(A)$) to ($2*(A)-(B)$);
	\draw ($2*(B)-(C)$) to ($2*(C)-(B)$);
	\node at (0,-2.5) {$(+,+,+)$};
	\node at (5.5,0) {$(-,+,+)$};
	\node at (7.5,-6.5) {$(-,+,-)$};
	\node at (0,-8.5) {$(+,+,-)$};
	\node at (-8,-6.5)  {$(+,-,-)$};
	\node at (-6,0) {$(+,-,+)$};
	\node at (0,4) {$(-,-,+)$};
	\foreach \x[evaluate={\y=1-\x}] in {.125, .25, ..., 1}
		{
\draw[->=.5,green!75!black] ($1.75*(C) -.75*\x*(A) -.75*\y*(B)$) -- (C);
\draw[->=.5,green!75!black] ($1.75*(B) -.75*\x*(A) -.75*\y*(C)$) -- (B);
\draw[->=.5,green!75!black] ($1.75*(A) -.75*\x*(C) -.75*\y*(B)$) -- (A);
}
	\foreach \x[evaluate={\y=1-\x}] in {0.25, .3125, ..., .75}
		{
\draw[->=.5,green!75!black] ($1.75*\x*(C) -.75*(A) + 1.75*\y*(B)$) -- ($\x*(C)+ \y*(B)$);
\draw[->=.5,green!75!black] ($1.75*\x*(A) -.75*(B) + 1.75*\y*(C)$) -- ($\x*(A)+ \y*(C)$);
\draw[->=.5,green!75!black] ($1.75*\x*(A) -.75*(C) + 1.75*\y*(B)$) -- ($\x*(A)+ \y*(B)$);
}
	\node[right=2pt] at (C) {(0,0,1)};
	\node[above right] at (B) {(0,1,0)};
	\node[above left] at (A) {(1,0,0)};
\end{tikzpicture}
\caption{Triangle $S_1$}
\label{fig:triangulo}
\end{center}
\end{figure}

As we can see in Figure \ref{fig:triangulo}, the hypersurfaces defined by $\{x_i=0\}$ divide the space in $2^{n+1}-1$ sections. We denote each of these sections as $(\pm,\dots,\pm)$ indicating which coordinates are greater or equal to $0$ and which are lower or equal to $0$. For example, the triangle corresponding to $S_1$ would be $(+,\dots,+)$.
Note that there is no section $(-,\dots,-)$.

The retraction in $(+,\dots,+)$ is the identity for any value of $t$. For the rest of the parts, we define a retraction:
\begin{equation*}
\begin{tikzcd}[row sep=0pt,/tikz/column 1/.append style={anchor=base east}, /tikz/column 2/.append style={anchor=base west}]
(M_1\cap\RR^{n+1})\times{[0,1]}\rar["r_2"]&M_1\cap\RR^{n+1}\times [0,1]\\
(Q, t)\rar[mapsto]& t P_Q + (1 - t) Q,
\end{tikzcd}
\end{equation*}
where $P_Q$ is a point depending on $Q\in M_1\cap\RR^{n+1}$. 

In general, for a point $Q$ in a region $(\pm,\dots,\pm)$, we want the retraction to take the negative coordinates to $0$ and to change the positive ones such that at the end the sum is $1$. Therefore, we can write the point $P_Q$ as:
\begin{equation*}
P_Q=\left(\frac{x_0\chi_0}{\displaystyle 1-\sum_{j=0}^n x_j(1-\chi_j)}, \dots,
\frac{x_n\chi_n}{\displaystyle 1-\sum_{j=0}^n x_j(1-\chi_j)} \right)
\end{equation*}
where $\chi_j=1$ when $x_j\geq0$ and $\chi_j=0$ otherwise. It is easy to check that the points $tP_Q+(1-t)Q$ are in $M_1$ for all $t$. To see this, we remark that $\sum x_i=1$ implies that $\sum x_i\chi_i+\sum x_i(1-\chi_i)=1$ and therefore $\sum x_i\chi_i=1-\sum x_i(1-\chi_i)$. Then
\begin{equation*}
\sum_{i=0}^n \left((1-t)x_i+\frac{tx_i\chi_i}{\displaystyle 1-\sum_{i=0}^n x_j(1-\chi_j)}\right)=
(1-t)\sum_{i=0}^n x_i+\frac{t}{\displaystyle 1-\sum_{i=0}^n x_j(1-\chi_j)}\sum_{i=0}^n x_i\chi_i=1.
\end{equation*}

Then, we have proved the following result:
\begin{lemma}
The map $r_2$ is a strong deformation retraction from $M_1\cap \RR^{n+1}$ to $S_1$.
\end{lemma}

Now that we have the retractions that allow us to retract $M_1$ into $S_1$, we will lift them in order to get a retraction from $M_d$ to $S_d$.

\begin{theorem}
	The space $S_d$ is a strong deformation retract of $M_d$.
\end{theorem}
\begin{proof}
We want to produce the following diagram:
\begin{equation*}
\begin{tikzcd}
	(x_0,\dots,x_n)\dar[mapsto] & [-30pt]	M_d \rar[shift left=2pt,harpoon] \dar & \lar[shift left=2pt,harpoon] M_d\cap\{x^d_i\in\RR,\,\forall i\}\rar[shift left=2pt,harpoon]\dar & \lar[shift left=2pt,harpoon]S_d \dar\\
	(x^d_0,\dots,x^d_n)	&M_1 \rar[shift left=2pt,harpoon] & \lar[shift left=2pt,harpoon] M_1\cap\RR^{n + 1}\rar[shift left=2pt,harpoon] & \lar[shift left=2pt,harpoon]S_1
\end{tikzcd}
\end{equation*}
where the horizontal arrows represent retractions. That is, we want to lift the retractions built for $M_1$ and $S_1$ to $M_d$ and $S_d$ via this branched cover.

We start seeing how the lifting of $r_1$ works. Since $r_1$ acts componentwise, we can write it as $r_1((x_0,\dots,x_n),t)=(\bar{r}(x_0,t),\dots,\bar{r}(x_n,t))$ with $\sum\bar{r}(x_i,t)=1$ whenever $(x_0,\dots,x_n)\in M_1$. Then, it suffices to find $G(x,t)$ such that $G(x,t)^d=\bar{r}(x^d,t)$.


In order to see this, we can write an explicit expression of $\bar{r}$. For a point $x\in\CC$ the action of $\bar{r}$ is the projection over the real axis, as seen in Figure \ref{realprojection}.
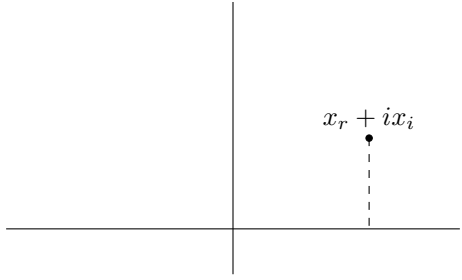
\begin{figure}[h]
\begin{center}
\begin{tikzpicture}[scale=0.6]
	\draw (-5,0) to (5,0);
	\draw (0,5) to (0,-1);
	\draw[fill] (3,2) circle (2pt) node[above]{$x_r+ix_i$};
	\draw[dashed] (3,2) to (3,0) node[below] {$x_r$};
\end{tikzpicture}
\end{center}
\caption{Projection over the real axis}
\label{realprojection}
\end{figure}
We can then write the retraction as $\bar{r}(x,t)=m(x,t)e^{2i\pi\theta(x,t)}$ where $m(x,t)$ is the function indicating how the modulus of $\bar{r}(x,t)$ changes and $\theta(x,t)$ portrays the change of the argument. Let us give explicit formulas.

For a $x=x_r+ix_i\in\CC$, $\bar{r}(x,t)=x_r+(1-t)ix_i$, we can write 
\[
m(x,t)=+\sqrt{x_r^2+(1-t)^2x_i^2} \text{ and } \theta(x,t)=\arctan\frac{(1-t)x_i}{x_r}.
\]
We can divide $\CC$ in $2d$ sectors of angle $\frac{\pi}{d}$ as in Figure~\ref{divisionC}, where each of the sections will be homeomorphic to either the upper half complex plane or to the lower half one.

\begin{figure}[h]
	\begin{center}
		%
		\begin{tikzpicture}
		\foreach \x in {0, ..., 5}{
			\draw (0,0) to (360/6*\x:3);}
		\draw[->] (1.5,0) to[out=90, in=350] (0.6,1.1) node[ xshift=1.0cm, yshift=-0.2cm]{$\pi/d$};
		\end{tikzpicture}

	\end{center}
	\caption{Division of the complex plane}
	\label{divisionC}
\end{figure}
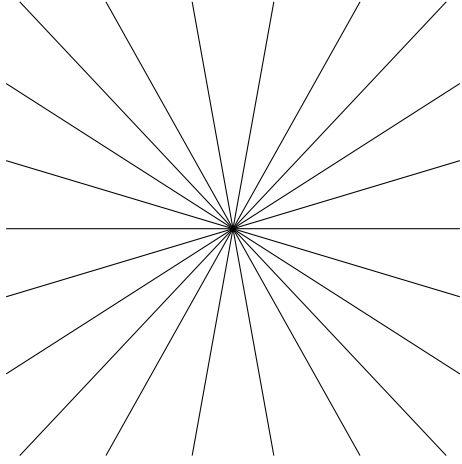

Then we can get homeomorphisms from the upper complex plane to the sectors of angle $[(k-1)\frac{\pi}{d}, k\frac{\pi}{d}]$ with $k=1,3,...,2d-1$, and from the lower complex plane to the sectors of angle $[(k-1)\frac{\pi}{d}, k\frac{\pi}{d}]$ with $k=2,4,...,2d$, which we will denote $f_k$. The lifting of $r_1$ will then be given by the $G$ previously described. This $G$ will act by taking the point $(x_0,\dots,x_n)\in M_d$ and mapping it to the upper or lower complex plane via the inverse homeomorphisms $f^{-1}_k$, then taking $r_1$ to take it to the real line, and finally bringing it back to its original sector with the homeomorphism~$f_k$.

The homeomorphisms are given by
\begin{align*}
	f_k(|z|\exp^{i\theta})&=|z|^\frac{1}{d}\exp{\left(i\frac{\theta}{d}+(k-1)i\frac{\pi}{d}\right)}\text{ for }k=1,3,5,\dots,2d-1,\\
	f_k(|z|\exp^{i\theta})&=|z|^\frac{1}{d}\exp{\left(i\frac{\theta-\pi}{d}+(k-1)i\frac{\pi}{d}\right)}\text{ for }k=2,4,6,\dots,2d.
\end{align*}
We will also need the inverses of these homeomorphisms, which are:
\begin{align*}
	f_k^{-1}(|z|\exp^{i\theta})&=|z|^d\exp{\left(id\left(\theta-(k-1)\frac{\pi}{d}\right)\right)}\text{ for }k=1,3,5,\dots,2d-1,\\
	f_k^{-1}(|z|\exp^{i\theta})&=|z|^d\exp{\left(id\left(\theta-\pi-(k-1)i\frac{\pi}{d}\right)\right)}\text{ for }k=2,4,6,\dots,2d.
\end{align*}
Then, we define $G=f_k\circ\bar{r}\circ f_k^{-1}$, where $k$ depends on the point $z$ to which we will apply it. For $z\in\CC$, given by $z=|z|\exp^{i\theta}$ for $\theta\in[(k-1)\frac{\pi}{d}, k\frac{\pi}{d}]$, we have:
\begin{equation}
\begin{aligned}
	G(z,t)&=f_k\left(\bar{r}(f_k^{-1}(z),t)\right)\\
	&=f_k\left(\bar{r}(|z^d|\exp^{id(\theta-(k-1)\frac{\pi}{d})},t)\right)\\
	&=f_k\left(\sqrt{(z^d)_r^2+(1-t)^2(z^d)_i^2}\, \exp{\left(i\arctan\left((1-t)\tan\left(d\left(\theta-(k-1)\frac{\pi}{d}\right)\right)\right)\right)}\right)\\
	&=\sqrt[2d]{(z^d)_r^2+(1-t)^2(z^d)_i^2}\,
	\exp{\left(\frac{i}{d}\arctan\left((1-t)\tan\left(d\theta\right)\right)\right)}
\end{aligned}
\label{Gliftedr1}
\end{equation}
Then, we can see that $G(z,t)^d=\bar{r}(z^d,t)$, since:
\begin{equation*}
\begin{tikzcd}[column sep=0pt, row sep=5pt]
G(z,t)^d\dar[equal, equal]&=&m(z,t)^d\exp{\left(i\arctan((1-t)\tan d\theta)\right)}\\
\bar{r}(z^d,t)&=&m(z^d,t)\exp{(i\arctan((1-t)\tan d\theta))}\
\end{tikzcd}
\end{equation*}
We can then write the lifting of $r_1$ as $\tilde{r}_1((x_0,\dots,x_n),t)=(G(x,t), \dots, G(x_n,t))$.

Lastly, we are going to see how to lift $r_2$ in order to get a retraction from $M_d\cap\{x^d_0, y^d, x^d_n\in\RR\}$ into $S_d$.
For a point $P=(x_0\xi_0,\dots,x_n\xi_n)\in M_d\cap\{x^d_0, y^d, x^d_n\in\RR\}$ (where $(x_0,\dots,x_n)$ are real and $\xi_i$ are $d$-th roots of unity), the corresponding point in $M_1\cap\RR^{n+1}$ is $(x^d_0,\dots,x^d_n)$ and therefore the goal point of the retraction $r_2$ will be
\begin{equation}
P^d_Q=\left(\frac{x^d_0\chi_0}{1-\sum x^d_j(1-\chi_j)}, \dots,\frac{x^d_n\chi_n}{1-\sum x^d_j(1-\chi_j)} \right)
\end{equation}
In the $(+,\dots,+)$ case, $r_2$ is the constant path, and therefore the lift $\tilde{r}_2$ will also be the constant path.

In the rest of the cases, the lift will be given by $\tilde{r}_2((x_0\xi_0,\dots,x_n\xi_n), t)$ equals
\[
\left(\sqrt[d]{(1-t)x^d_0+t\frac{x^d_0\chi_0}{1-\sum x^d_j(1-\chi_j)}}\xi_0,\dots,\sqrt[d]{(1-t)x^d_n+t\frac{x^d_n\chi_n}{1-\sum x^d_j(1-\chi_j)}}\xi_n\right).
\]

We have then proved that $S_d$ is a strong deformation retract of $M_d$.
\end{proof}

\begin{theorem}
 The strong retraction from $M_d$ to $S_d$ induces a strong retraction, and therefore, a homotopy equivalence between $\CC\PP^n\setminus C_d$ and $K$.
\end{theorem}
\begin{proof}
	The path followed by each point in $\CC\PP^n\setminus C_d$ in the retraction will be given by the path followed by a representative in $M_d$ as it is retracted to $S_d$. The class representatives we choose in $\CC\PP^n\setminus C_d$ will be such that the coordinates to the $d$ power add up to $1$. Since this choice is not unique, we are now going to see that if we apply the retraction $r=r_2\circ r_1$ to two such representatives, they are sent to the same class by $r$.

	Let $p=[x_0:\dots:x_n]\in\CC\PP^n\setminus C_d$ and $\lambda p=[\lambda x_0:\dots: \lambda x_n]$ be two representatives of some point $p$, since both representatives must be such that $\sum x^d_i=1$, necessarily $\lambda$ is a $d-$th root of unity.

	First, we see how $r_1$ acts on these representatives.
	\begin{align*}
			\tilde{r}_1(p,t) & =  [G(x_0,t): \dots: G(x_n,t)], \\
			\tilde{r}_1(\lambda p,t) & =  [G(\lambda x_0,t): \dots: G(\lambda x_n,t)].
	\end{align*}
	If we define $\theta(x,t)=\arctan\frac{(1-t)x_i}{x_r}$ for $x\in\CC$, we get that
	\begin{equation*}
		\theta(\lambda x,t)=\arctan\frac{(1-t)\lambda x_i}{\lambda x_r}=\arctan\frac{(1-t)x_i}{x_r}=\theta(x,t)
	\end{equation*}
	By \eqref{Gliftedr1}:
	\begin{align*}
			G(x,t) & =  \sqrt[2d]{(x^d)_r^2+(1-t)^2(x^d)_i^2}\exp{\left(\frac{i}{d}\arctan((1-t)\tan(d\theta)))\right)}  \\
			G(\lambda x,t) & =  \sqrt[2d]{(\lambda^d x^d)_r^2+(1-t)^2(\lambda^d x^d)_i^2}\exp{\left(\frac{i}{d}\arctan((1-t)\tan(d\theta)))\right)} \\
	\end{align*}
	Since $\lambda$ is a $d-$th root of unity, we get that $G(x,t)=G(\lambda x,t)$. Therefore, we get that $\tilde{r}_1(p,t)=\tilde{r}_1(\lambda p,t)$. It's just left to check that $r_2$ also sends them to the same class.

	If $[x^d_0: \dots: x^d_n]$ is in $(+,\dots,+)$, the retraction is the identity, and therefore both representatives will stay the same and therefore go to the same class.
	
	For the rest of the  regions, we have
\begin{equation*}
\tilde{r}_2([x_0\xi_0:\dots:x_n\xi_n], t) =\left[H(x_0,\chi_0)\xi_0:\dots:H(x_n,\chi_n) \xi_n\right]
\end{equation*}
where
\[
H(x_i, \chi_i):=\sqrt[d]{(1-t)x^d_i+t\frac{x^d_i\chi_i}{1-\sum x^d_j(1-\chi_j)}},\quad i=0,\dots,n.
\]
Then, for the point $\lambda p$ we get
\begin{align*}
\tilde{r}_2([\lambda x_0\xi_0:\dots:\lambda x_n\xi_n], t)&=\left[H(\lambda x_0, \chi_0)\xi_0:\dots:H(\lambda x_n, \chi_n)\xi_n\right]\\ 
&=\left[H( x_0, \chi_0)\xi_0:\dots:H(x_n, \chi_n)\xi_n\right].
\end{align*}
So we see that $\tilde{r}_2$ also takes both representatives will end to the same point.
\end{proof}

We have now seen that there is a homotopy equivalence between $\CC\PP^n\setminus C_d$ and $K$. Therefore, to get a CW-complex structure on $\CC\PP^n\setminus C_d$, it suffices to get it for $K$.

We denote as $\Delta$ the $\Delta-$complex $[v_0, \dots, v_n]$ that has $n+1$ vertices, and $n+1-j$ $j$-simplices. Then, we can identify $S_1$ with $\Delta$, and we get a diagram:

\begin{equation*}
\begin{tikzcd}
\Delta \rar[leftrightarrow]& S_1&\\ 
\Delta^d\uar\rar[leftrightarrow]&S_d\uar\rar[hookrightarrow]\dar& M_d\dar\\ 
&K \rar[hookrightarrow]& \PP^n\setminus C_d,
\end{tikzcd}
\end{equation*}
where $\Delta^d$ is a $\Delta-$complex that we are trying to find and identify with $S_d$, equivalently as what we have done for $d=1$. Since $S_d$ is a ramified cover of $S_1$ by the power to the $d$ map, we want $\Delta^d$ to be a ramified cover of $\Delta$ by the same map.



First, we give a detailed description of $\Delta$. We will denote the $0$-simplices, the $i$-simplices and the $n$-simplex as:
\begin{align*}
	X^j & =  v_j  =  e_j \\
	X^{(j_0,\dots,j_i)} & =  [v_{j_0},\dots,v_{j_i}]\\
	X & =  [v_0,\dots,v_n]
\end{align*}
The face maps of $X$ and $X^{(j_0,\dots,j_i)}$ are:
\begin{equation*}
	\delta^i_k(X^{(j_0,\dots,j_i)}) =[v_{j_0},\dots,\hat{v},\dots,v_{j_i}]
\end{equation*}
where $\hat{v}$ indicates that we take out the vertex in position $k$, for $k\leq i$.
And the boundary maps are given by:
\begin{equation*}
\partial^i =\sum (-1)^{k}\delta^i_k
\label{bordermaps}
\end{equation*}
We define now the $\Delta$-complex
$\Delta^d$. It has $(n+1)d$ $0-$simplices, ${n+1}\choose{2}$ $d^2$ $1-$simplices, ${n+1}\choose{i}$ $d^{i+1}$ $i-$simplices and $d^{n+1}$ $n-$simplices.

Let $\mu_d=\{\xi| \xi^d=1\}$. We can express $S_d$ as a union
\begin{equation}
	\begin{array}{rll}
		S_d & = & \{(x_0,\dots,x_n)| \sum x^d_i=1 \text{ and } x^d_i\in\RR_{\geq0}\} \\
		&  = & \displaystyle\bigcup_{(\xi_0, \dots, \xi_n)\in\mu^{n+1}_d} X_{(\xi_0, \dots, \xi_n)} \\
	\end{array}
\end{equation}
where $X_{(\xi_0, \dots, \xi_n)}=\{(x_0,\dots,x_n)| \sum x^d_i=1 \text{ and } \xi_0x, \dots, \xi_nx_n\geq0\}$.

Then, we can see that the faces of the simplex are given by $\mathring{X}_{(\xi_0, \dots, \xi_n)}$, while the $j-$simplices are given by 
\[
X_{(\xi_0, \dots,\xi_n)}\cap\{x_{i_1},\dots,x_{i_j}=0\},
\]
 where $\xi_{i_1},\dots,\xi_{i_j}=1$, and the vertices by 
\[
X_{(1,\dots,\xi_i, 1, 1)}\cap\{x_j=1,\,\forall j\neq i\}.
\]
We can remark that $X_{(\xi_0, \dots, \xi_n)}=(\xi_0, \dots, \xi_n)\cdot X_{(1,\dots,1)}$, and therefore that we have a group action of $\mu^{n+1}_d$ over $S_d$.
Let
\begin{align*}
	X^{(j_0,\dots,j_i)}_{(1,\dots,1)} & =  X_{(1, \dots, 1)}\cap\{x_{j_k}\neq0,\, \forall k=0,\dots i\text{ and }x_{i}=0\text{ for the rest}\}\\
	X^j_{(1,\dots,1)} & =  X_{(1, \dots, 1)}\cap\{x_i=0,\,\forall i\neq j\}\\
\end{align*}
We can describe the chain complexes. They are abelian groups but they are also $R_d$-complexes
where $R_d$ is the group algebra $\ZZ[\mu_d^{n+1}]$. For the $n$-chains we have
\[
C_{n}(\Delta^d) = R_d\langle X_{(1,\dots,1)}\rangle,
\]
which is a free principal $R_d$-module, i.e. a free abelian group of rank~$d^{n+1}$. For the $i$-chains, with $i<n$ we deal with non-free modules.
Let us denote as $t_0, \dots, t_n$ the variables of $R_d$; note that $t_0^d=\dots=t_n^d=1$.
Then 
\[
C_i(\Delta^d) = \bigoplus_{(t_{j_1},\dots,t_{j_{n-i}})}\frac{R_d}{(t_{j_1} - 1,\dots,t_{j_{n-i}} - 1)} \langle X^{(k_0,\dots,k_i)}_{(1,\dots,1)}\rangle,
\]
where $t_{j_s}\neq t_{k_r}\forall s,r$ i.e, a free abelian group of rank $d^i$. Finally the $0$-chains are
\[
C_0(\Delta^d) = \bigoplus_{t_j}\frac{R_d}{(t_0 - 1,\dots,t_{j-1}-1,t_{j+1}-1,\dots, t_n - 1)} \langle X^j_{(1,\dots,1)}\rangle
\]
a free abelian group of rank~$(n+1)d$.

The face maps are $R_d$-morphisms defined as 
\begin{equation*}
\begin{aligned}
	\delta^n_i(X_{(1,1,1)}) & =X^{(0,\dots,i-1,i+1,\dots,n)}_{(1,\dots,1)} &&& \delta^j_i(X^{(k_0,\dots,k_j)}_{(1,\dots,1)})&= X^{(k_0,\dots,k_{i-1},k_{i+1},\dots,k_n)}_{(1,\dots,1)} .
\end{aligned}
\end{equation*}
%
The boundary maps are defined as in \eqref{bordermaps}.

We have then defined $\Delta^d$, the simplicial complex with which we can identify $S_d$. Our last step now is to complete the diagram with a simplicial complex associated to the projective space $K$, which we will call $\Delta^d_\textrm{proj}$.
\begin{equation*}
\begin{tikzcd}
\Delta \rar[leftrightarrow]& S_1&\\ 
\Delta^d\uar\rar[leftrightarrow]\dar&S_d\uar\rar[hookrightarrow]\dar& M_d\dar\\ 
\Delta^d_\textrm{proj}\rar[leftrightarrow]&K \rar[hookrightarrow]& \PP^2\setminus C_d,
\end{tikzcd}
\end{equation*}
The map $\Delta^d\longrightarrow\Delta^d_\text{proj}$ is given by $(x_0,\dots,x_n)\longmapsto[x_0:\dots:x_n]$. Since we are passing to the projective plane, two points $(x_0\xi_0, \dots, x_n\xi_n)$ and $(x_0\xi_0\xi, \dots, x_n\xi_n\xi)$ will have the same image.

The $n$-cells $X_{(\xi_0, \dots, \xi_n)}$ and $X_{(\xi_0\xi, \dots, \xi_n\xi)}$ of $\Delta^d$, project the same $n$-cell of $\Delta^d_\textrm{proj}$, which will be denoted as $X_{[\xi_0: \dots: \xi_n]}$. 
The same convention is followed for the rest of $i-$simplices. Note that for the vertices, since all coordinates but one are $0$, we have only $n+1$ vertices $X^0, \dots, X^n$.

The group $\mu_d^{n+1}$ also acts, but the action of the diagonal vectors is trivial. Let us denote 
\[
T_d:=\ZZ\left[\frac{\mu_d^{n+1}}{\mu_d\langle(\xi,\dots,\xi)\rangle}\right];
\]
we use variables $t_0,\dots,t_n$ as for $R_d$ and we have to add the relation $t_0 \cdots t_n=1$.
This group is actualy isomorphic to $\mu_n^d$ but it is more convenient to use this presentation.
We have the following chain complexes:
\begin{align*}
C_n(\Delta^d_\text{proj}) &= T_d\langle X_{[1:\dots:1]}\rangle,\\
C_i(\Delta^d_\text{proj}) &= \bigoplus_{(t_{j_1},\dots,t_{j_{n-i}})}\frac{T_d}{(t_{j_1} - 1,\dots,t_{j_{n-i}} - 1)} \langle X^{(k_0,\dots,k_i)}_{[1:\dots:1]}\rangle,\\ 
C_0(\Delta^d) &= \bigoplus_{t_j}\frac{T_d}{(t_0 - 1,\dots,t_{j-1}-1,t_{j+1}-1,\dots, t_n - 1)} \langle X^j_{[1:\dots:1]}\rangle.
\end{align*}
%

\begin{example}{}
	
	We are going to compute now the simplicial complex for $S_d$ and $K$ when $d=n=2$,  this is, our curve will be $C_D=\{[x:y:z]\in\PP^2|x^2+y^2+z^2=0\}$, $S_2=\{(x,y,z)\in M_2| x^2, y^2, z^2\in\RR_{\geq 0}\}$ and $K=\{[x:y:z]\in\PP^2|x^2,y^2,z^2\in\RR_{\geq0}\}$.

	We start with $S_2$. First, we can observe that, from all points in $\CC$, the only ones whose square power is real are either the real numbers of the pure imaginary numbers. And from these, the later have negative square power. Therefore, we can conclude that $S_2=\{(x,y,z)\in\RR^3| x^2+y^2+z^2=1\}$, that is, $S_2$ is the real three dimensional sphere. We now see the simplicial complex $\Delta^2$ that can be associated to it such that $\phi(\Delta^2)=\Delta$.

	In this case, we have that $\mu_2=\{\xi|\xi^2=1\}=\{1,-1\}$. Then, we have that $X_{(1,1,1)}=\{(x,y,z)\in\RR^3_{\geq0} |x^2+y^2+z^2=1\}$, this is, $X_{(1,1,1)}$ is the set of points given by the positive square roots of the points in $S_1$. And the rest of $2-$simplices will be given by the action of $\mu^3_2$ on $X_{(1,1,1)}$, which will give us the $2-$simplices	$X_{(1,1,-1)}$, $X_{(1,-1,1)}$, $X_{(-1,1,1)}$, $X_{(1,-1,-1)}$, $X_{(-1,1,-1)}$, $X_{(-1,-1,1)}$ and $X_{(-1,-1,-1)}$.

	The face maps of $X_{(1,1,1)}$ are:
	\begin{equation*}
\begin{tikzcd}[row sep=0, column sep=8pt]
			\delta^2_0(X_{(1,1,1)}) \rar[equal] & L^x_{(1,1,1)} \rar[equal] & \{(0,\sqrt{1-t},\sqrt{t})|t\in[0,1]\}\\
			\delta^2_1(X_{(1,1,1)})\rar[equal]& L^y_{(1,1,1)} \rar[equal] & \{(\sqrt{1-t},0, \sqrt{t})|t\in[0,1]\}\\
			\delta^2_2(X_{(1,1,1)}) \rar[equal] & L^z_{(1,1,1)} \rar[equal] & \{(\sqrt{1-t},\sqrt{t},0)|t\in[0,1]\}
\end{tikzcd}
	\end{equation*}
	By applying the group action as shown before we get the faces of the rest of $2-$simplices.

	For the $1-$simplices, the group acting would be $\mu^3_2/\mu_2$, since one of the  coordinates is $0$. Therefore, the $1-$simplices will be $L^x_{(1,1,1)}$, $L^x_{(1,1,-1)}$, $L^x_{(1,-1,1)}$, $L^x_{(1,-1,-1)}$, $L^y_{(1,1,1)}$, $L^y_{(1,1,-1)}$, $L^y_{(-1,1,1)}$, $L^y_{(-1,1,-1)}$, $L^z_{(1,1,1)}$, $L^z_{(1,-1,1)}$, $L^z_{(-1,1,1)}$ and $L^z_{(-1,-1,1)}$.

The face maps of $L^z_{(1,1,1)}$, $L^y_{(1,1,1)}$ and $L^x_{(1,1,1)}$ are:
\begin{equation*}
\begin{tikzcd}[row sep=0, column sep=8pt]
	\delta^1_0(L^z_{(1,1,1)}) \rar[equal]  & V^y_{(1,1,1)} \rar[equal]  & (1,0,0)\\
	\delta^1_1(L^z_{(1,1,1)}) \rar[equal]  & V^z_{(1,1,1)} \rar[equal]  & (0,1,0)\\
	\delta^1_0(L^y_{(1,1,1)}) \rar[equal]  & V^x_{(1,1,1)} \rar[equal]  & (1,0,0)\\
	\delta^1_1(L^y_{(1,1,1)}) \rar[equal]  & V^z_{(1,1,1)} \rar[equal]  & (0,0,1)\\
	\delta^1_0(L^x_{(1,1,1)}) \rar[equal]  & V^x_{(1,1,1)} \rar[equal]  & (0,1,0)\\
	\delta^1_1(L^x_{(1,1,1)}) \rar[equal]  & V^y_{(1,1,1)} \rar[equal]  & (0,0,1)\\
\end{tikzcd}
\end{equation*}

Lastly, the $0-$simplices will be $V^x_{(1,1,1)}=(1,0,0)$, $V^x_{(-1,1,1)}=(-1,0,0)$, $V^y_{(1,1,1)}=(0,1,0)$, $V^y_{(1,-1,1)}=(0,-1,0)$, $V^z_{(1,1,1)}=(0,0,1)$ and $V^z_{(1,1,-1)}=(0,0,-1)$. In Figure \ref{fig:S2} we can see the  $\Delta-$complex that we associate to $S_2$ which corresponds to the partition of the $2$-spheres in triangles with right angles.
\begin{figure}[!hp]
	\begin{center}
		\begin{tikzpicture}[scale=0.7]
		\draw[gray, thin] (-5,0) to (5,0);
		\draw[gray, thin] (0,-5) to (0,5);
		
		\draw[gray, thin] (-4,-4) to (4,4);
		\draw[fill, red] (4,0) circle (2pt);
		\draw[fill, red] (-4,0) circle (2pt);
		\draw[fill, green] (0,4) circle (2pt);
		\draw[fill, green] (0,-4) circle (2pt);
		\draw[fill, blue] (1.41,1.41) circle (2pt);
		\draw[fill, blue] (-1.41,-1.41) circle (2pt);
		
		\draw (0,0) circle (4cm);
		
		\draw (-4,0) arc (180:360:4 and 1.5);
		\draw[dashed] (4,0) arc (0:180:4 and 1.5);
		\draw[dashed] (0,4) arc (90:270:-1.5 and 4);
		\draw (0,-4) arc (270:90:1.5 and 4);
		
		\draw[pink, thick](-4,0.02) arc (180:250:4 and 1.5);
		\draw[pink, thick](-4.02,0.02) arc (180:90:4);
		\draw[pink, thick](0,4) arc (270:90:1.5 and -4);
		
		\draw[purple, thick](-4,-0.02) arc (180:250:4 and 1.5);
		\draw[purple, thick](-4,-0.02) arc (180: 270:4);
		\draw[purple, thick] (-1.43,-1.41) arc (200:270:1.5 and 3.9);
		
		\draw[yellow, thick, dashed](-3.98,0.02) arc (180:70:4 and 1.5);
		\draw[yellow, thick, dashed](-3.98,0.02) arc (180: 90:4);
		\draw[yellow, thick, dashed] (1.39,1.41) arc (200:270:-1.5 and -3.9);
		
		\draw[orange, thick, dashed](-3.98,-0.02) arc (180:70:4 and 1.5);
		\draw[orange, thick, dashed](-3.98,-0.02) arc (180: 270:4);
		\draw[orange, thick, dashed] (1.39,1.41) arc (160:270:-1.5 and 4);
		
		\draw[teal, thick, dashed](4.02,0.02) arc (180:110:-4 and 1.5);
		\draw[teal, thick, dashed](4.02,0.02) arc (180: 270:-4);
		\draw[teal, thick, dashed] (1.42,1.41) arc (200:270:-1.5 and -3.9);
		
		\draw[violet, thick](4.02,-0.02) arc (0:110:4 and -1.5);
		\draw[violet, thick](4.02,-0.02) arc (180: 90:-4);
		\draw[violet, thick] (-1.39,-1.42) arc (200:270:1.5 and 3.9);
		
		\draw[magenta, thick, dashed](3.98,0.02) arc (180:110:-4 and 1.5);
		\draw[magenta, thick, dashed](3.98,-0.02) arc (180: 90:-4);
		\draw[magenta, thick, dashed] (1.43,1.41) arc (160:270:-1.5 and 4);
		
		\draw[lime, thick, dashed](4.02,0.02) arc (0:110:4 and -1.5);
		\draw[lime, thick, dashed](3.98,0.02) arc (180: 270:-4);
		\draw[lime, thick, dashed] (0.02,3.98) arc (270:160:1.5 and -4);
		\end{tikzpicture}
	\end{center}
	\caption{Division in simplices of $S_2$}
	\label{fig:S2}
\end{figure}
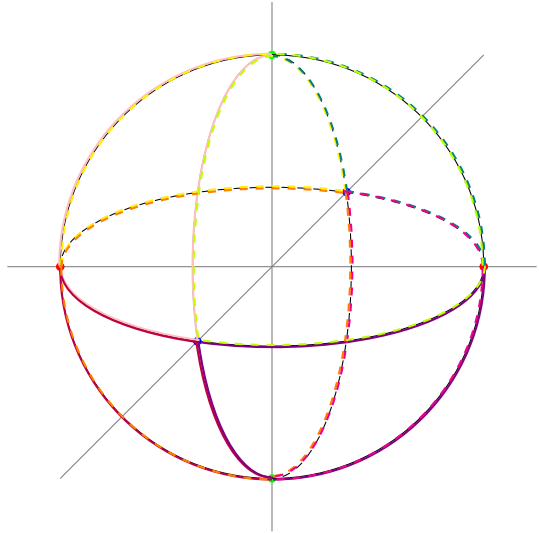




Now, we give a description of $\Delta^2_\text{proj}$. First, for the $0-$simplices, we have that $V^x_{[1,1,1]}=V^x_{[-1,1,1]}=[1:0:0]$, $V^y_{[1,1,1]}=V^y_{[1,-1,1]}=[0:1:0]$ and $V^z_{[1,1,1]}=V^z_{[1,1,-1]}=[0:0:1]$.

We will have $6$ $1-$simplices, obtained from the $12$ of $\Delta^d$ by identifying: $L^x_{[1:1:1]}=L^x{[1:-1:-1]}$, $L^x_{[1:-1:1]}=L^x{[1:1:-1]}$, $L^y_{[1:1:1]}=L^y{[1:-1:-1]}$, $L^y_{[1:-1:1]}=L^y{[1:1:-1]}$, $L^z_{[1:1:1]}=L^z{[1:-1:-1]}$ and $L^z_{[1:-1:1]}=L^z{[1:1:-1]}$.

Lastly, we will have $4$ $2-$simplices, $X_{[1:1:1]}=X_{[-1:-1:-1]}$, $X_{[1:1:-1]}=X_{[-1:-1:1]}$, $X_{[1:-1:1]}=X_{[-1:1:-1]}$ and $X_{[-1:1:1]}=X_{[1:-1:-1]}$.

The face maps will be obtained by passing to the projective the face maps of $\Delta^d$.
\end{example}

\bibliographystyle{amsplain}
\bibliography{biblio}

%
%
%

\appendix
\section{SageMath code}

We have written a SageMath~\cite{Sagemath} function allowing us to define the $\Delta$-complex associated to the complement of a projective Fermat hypersurface of degree $d$ in $\PP^n$.

The code for the function is:
\begin{python}
def DeltaComplex_hypersurface(n,d):
    resdic = [{tuple(i*[0]+[1] + (n-i)*[0]): None for i in range(n+1) }]+ [{} for i in range(n)]

    for nzerosi in range(n+1): 
        for t in Tuples(range(d+1),n-nzerosi ):
            T = nzerosi*[0] + [1] + list(t)
            dimT = len(T) - T.count(0) - 2
            if dimT >= 0:
                caras = []
                for i in range(n+1):
                    if T[i] != 0:
                        CT = copy(T)
                        CT[i] = 0
                        canonicalize(CT)
                        caras.append(tuple(CT))
                resdic[dimT+1][tuple(T)] = caras
    
    DC = [[tuple() for _ in range(n+1)]]
    for i in range(1,n+1): 
        lis = []
        for p,v in resdic[i].items():
            lisp = []
            for s in v:
                lisp.append(list(resdic[i-1].keys()).index(s))
            lis.append(lisp)
        DC.append(lis)
    
    dc=DeltaComplex(DC) 
    return dc
\end{python}

This function uses an auxiliary function, named $canonicalize$ in the code. 
$canonicalize$ is used to identify the simplices in the projective space. The code to this function is:
\begin{python}
def canonicalize(t):  
    v = next(filter(lambda a : a != 0, t))
    correct = v - 1
    for i in range(len(t)):
        if t[i] != 0:
            t[i] -= correct
            if t[i] <= 0:
                t[i] += d
\end{python}

In particular this auxiliary function takes as input a face to a simplex and outputs its canonical representative in the projective space. We understand as canonical representative the element in the same projective class such that the first nonzero coordinate is equal to $1$.

\begin{example}{}
We now present an example of how to compute the $\Delta$-complex of a Fermat's hypersurface.

We take
\begin{python}
n = 4
d = 6
\end{python}

This is, we want to compute the $\Delta$-complex of the complement of the curve
\begin{equation}
 C=\{[x_0:x_1:x_2:x_3:x_4]\in\CC\PP^4 | x_0^6+x_1^6+x_2^6+x_3^6+x_4^6+=0]\}
\end{equation}

Running the code above we get:
\begin{python}
Delta_Complex = DeltaComplex_hypersurface(4,6)
 -> Delta complex with 5 vertices and 2802 simplices
\end{python}

Once we have gotten the $\Delta$-complex, we can compute the fundamental group or the cohomology of the space, for example.
\begin{python}
Dc = SimplicialSet(Delta_Complex)
dc_punto = it commit Dc.set_base_point(Dc.all_n_simplices(0)[0])
dc_punto.fundamental_group()
 -> Finitely presented group < e0 | e0^6 >
dc_punto.cohomology()
 -> {0: 0, 1: 0, 2: C6, 3: 0, 4: Z^520 x C6}
\end{python}
\end{example}
\end{document}